\documentclass[letterpaper,10pt]{amsart}
\usepackage{thmtools}
\usepackage{amssymb}
\usepackage{xcolor}
\usepackage{hyperref}

\definecolor{carminepink}{rgb}{0.92, 0.3, 0.26}

\hypersetup{colorlinks=true,allcolors=carminepink} 

\declaretheorem[parent=section]{theorem}
\declaretheorem[sibling=theorem]{proposition}
\declaretheorem[sibling=theorem]{lemma}
\declaretheorem[sibling=theorem]{corollary}

\declaretheorem[sibling=theorem,style=remark]{remark}

\declaretheorem[sibling=theorem,style=definition]{definition}
\newtheorem*{acknow*}{Acknowledgments}
\DeclareSymbolFont{bbold}{U}{bbold}{m}{n}
\DeclareSymbolFontAlphabet{\mathbbold}{bbold}

\providecommand{\xset}[1]{x_1,\ldots, x_{#1}}
\newcommand{\reals}{\mathbb{R}}					          

\newcommand{\Ex}{\mathbb{E}}				  
\newcommand{\E}{\mathbb{E}}            
\newcommand{\Prob}{\mathbb{P}}    
\newcommand{\sphere}{\mathbb{S}^d}				         
\renewcommand{\frak}{\mathfrak}   
    
\newcommand{\betac}{\beta_{j,k}}                         
                              
\newcommand{\barg}{\frac{\ell_q}{B^{2j}}}                          
\newcommand{\M}{\mathcal{M}}						         
\newcommand{\cubew}{\lambda_{j,k}}
\newcommand{\cubep}{\xi_{j,k}}

\newcommand{\ind}[1]{\mathbbold{1}{\bbra{#1}}}
\newcommand{\Lambdaj}{\Lambda_j}		

\newcommand{\Ltwo}{{L^2\bra{\M}}}
\newcommand{\Lp}{{L^p\bra{\M}}}
\newcommand{\Lr}{{L^r\bra{\M}}}
\newcommand{\Lq}{{L^q\bra{\M}}}
\newcommand{\besov}{\mathcal{B}_{r,q}^s}

\definecolor{electricultramarine}{rgb}{0.25, 0.0, 1.0}

\newcommand{\eigen}{\ll}

\providecommand{\cc}[1]{\overline{#1}}
\providecommand{\abs}[1]{\left\vert#1\right\vert}			
\providecommand{\norm}[1]{\left\Vert#1\right\Vert}			

\providecommand{\bra}[1]{\left(#1\right)}			      
\providecommand{\sbra}[1]{\left[#1\right]}			    
\providecommand{\bbra}[1]{\left\{#1\right\}}			      
\providecommand{\needlet}[1]{\psi_{j,k}\left(#1\right)}	
\providecommand{\bfun}[1]{b\left(#1\right)}				  
\newcommand{\sumk}{\sum_{k=1}^{K_j}}   
\newcommand{\sumq}{\sum_{q \in \Lambdaj} }
                 
\providecommand{\Mint}[1]{\int_{\M}{#1}}		  

\renewcommand{\ll}{\ell}							         
\renewcommand{\l}{\left}                                      
\renewcommand{\r}{\right}                                    
\newcommand{\Lm}{\mathcal{L}_{\M}}
\providecommand{\needlettensor}[2]{\psi_{j,k}^{\otimes #1}\left(#2\right)}	
\newcommand{\kerdef}{\sum_{k=1}^{K_j}\needlettensor{n}{\xset{n}}}

\newcommand{\Var}[1]{\operatorname{Var}\left(#1\right) }

\calclayout

\begin{document}
\bibliographystyle{alpha}

\author{Solesne Bourguin}
\address{Boston University, Department of Mathematics and Statistics, 111 Cummington Mall, Boston, MA 02215, USA}
\email{solesne.bourguin@gmail.com}
\author{Claudio Durastanti}
\address{Ruhr  University  Bochum,  Faculty  of  Mathematics,  D-44780  Bochum,  Germany}
\email{claudio.durastanti@gmail.com}
\thanks{C. Durastanti is supported by the German DFG grant \textit{GRK} 2131}
  
\title[On high-frequency limits of $U$-statistics over compact manifolds]{On high-frequency limits of $U$-statistics in Besov spaces over compact manifolds} 
\begin{abstract}
In this paper, quantitative bounds in high-frequency central limit theorems are derived for Poisson based $U$-statistics of arbitrary degree built by means of wavelet coefficients over compact Riemannian manifolds. The wavelets considered here are the so-called needlets, characterized by strong concentration properties and by an exact reconstruction formula. Furthermore, we consider Poisson point processes over the manifold such that the density function associated to its control  measure lives in a Besov space. The main findings of this paper include new rates of convergence that depend strongly on the degree of regularity of the control measure of the underlying Poisson point process, providing a refined understanding of the connection between regularity and speed of convergence in this framework.
\end{abstract}

\maketitle

\section{Introduction}\label{sec:intro} 

\subsection{Motivations and overview of the literature}
Poisson-based $U$-statistics are a central tool in mathematical statistics and stochastic geometry and have been the object of many recent investigations and developments, in particular in combination with stochastic calculus with respect to Poisson random measures and Malliavin calculus (see for instance, \cite{BouPec, lrp, lesmathias}). More recently, Poisson based $U$-statistics have been studied by means of wavelet methods, yielding new types of results, especially in the high-frequency limit (see, e.g., \cite{bdmp, bd1}). Wavelet systems have drawn a lot of interest, both from a theoretical as well as an applied point of view, more specifically in astrostatistics, where the main motivations for the results presented here originate. In this paper, we will focus on a second-generation class of wavelets, namely, the so-called needlets.  Needlet frames were originally introduced over the $d$-dimensional sphere in \cite{npw2,npw1}, and then have been generalized to different settings such has spin fiber bundles (see, for example, \cite{gelmar,gelmar2010}) and compact Riemannian manifolds (see, e.g., \cite{gm2,gelpes, knp,pesenson}). The concentration property in the spatial domain of the needlets (see \eqref{locapropofneedlets}) allows one to focus on any subset of a given manifold without having to take into account the entire structure. Some of the most remarkable statistical applications are discussed and studied in \cite{bkmpAoS,bkmpAoSb,cammar,dlm}.
\\~\\
The results obtained in \cite{bdmp,bd1} deal with Poisson $U$-statistics of order one and two, for which quantitative central limit theorem are derived in the high-frequency limit.  In both the linear and quadratic regimes, the control measure of the Poisson random measure was taken to be absolutely continuous with respect to the Lebesgue measure with a positive bounded density bounded away from zero. In this paper, we consider much more general $U$-statistics of arbitrary order, based on a Poisson random measure on a general compact manifold $\M$ of arbitrary dimension $d$ with an absolutely continuous control measure based on a density assumed to live in a given Besov space $\besov$ (a detailed introduction to Besov spaces and the interpretation of their parameters is given in Subsection \ref{sub:besov}). In particular, this assumption allows us to provide a fine analysis of the asymptotic behavior of these $U$-statistics in terms of the regularity of the density of the control measure. To our knowledge, these are the first results of this type in this framework and refine considerably not only the study of the asymptotics of such objects, but also the quantitative bounds one can obtain, explicitly exhibiting the relationship between speed of convergence and regularity of the control measure.
\\~\\
More concretely, let $\bbra{\psi_{j,k}\colon j \geq 0,\ k=1,\ldots, K_j}$ be a needlet frame with scale parameter $B$ over a compact manifold $\M$, where $K_j$ is the total number of needlets corresponding to a given resolution level $j$ (see Subsection \ref{sub:needlet} for a self-contained introduction to needlet frames on manifolds). In view of the concentration property in the spatial domain enjoyed by the needlet frame -- see \eqref{locapropofneedlets} -- needlets are neglectable outside of a subregion of $\M$, called pixel, of size proportional to $B^{-jd}$. Assume $N_t$ denotes a Poisson random measure on the manifold $\M$ with control measure $\mu_t(dx) := R_t f(x)dx$, where $R_t >0$ stands for the expected number of observations at time $t$ and where $f \in \besov$. By a slight abuse of notation, we also denote by $N_{t}$ the support of the Poisson random measure, and for any $p\geq 2$, $N_{t,\neq}^{p}$ will stand for the intersection of $N_{t}^{p}$ with the complement of the diagonal sets of $\reals^p$. 
\\~\\
The main objects under consideration here are defined as follows. Fix a natural number $n \geq 1$ and for any $j \geq 0$, define
\begin{equation*}
U_j\bra{t} = \sum_{\bra{x_1, \ldots, x_n} \in N_{t,\neq}^{n}}h_j\bra{x_1, \ldots, x_n},
\end{equation*}
where the kernel $h_j$ is the symmetric function given by 
\begin{equation*}
h_j\bra{x_1, \ldots ,x_n} := \sum_{k=1}^{K_j}\psi_{j,k}^{\otimes n}\bra{x_1, \ldots ,x_n}.
\end{equation*}
Furthermore, denote by $\sigma_j^2$ the variance of $U_j(t)$ and write $\widetilde{U}_{j}(t) := U_j(t)/\sigma_j^2$ for the normalized version of $U_j(t)$. 
\\~\\
\noindent The goal of this paper is to establish quantitative bounds for $U$-statistics of arbitrary degree $p\in \mathbb{N}$ built on wavelets coefficients evaluated on Poisson random fields taking values on a compact Riemannian manifold $\M$ of dimension $d \in \mathbb{N}$. Originally presented by W. Hoeffding in \cite{hoeffding48}, $U$-statistics have been since then extensively studied in the literature, see for instance the textbooks \cite{lee, vandervaart} for details and discussions. The use of these statistics is typically motivated by a large set of standard problems in statistical inference, as well as in the field of stochastic geometry. As mentioned in \cite{vandervaart}, interresting examples are test statistics such as Wilcoxon or Kendall's $\tau$ or Cram\'er statistic. We are interested in the so-called high-frequency limit, where the scale (i.e., the resolution level of the wavelet system) under investigation and the number of tests to be considered depend on the size of the observational sample. A central tool in the proof of our main results is the so-called Stein-Malliavin method for Poisson point processes, introduced in the seminal papers \cite{PSTU,PecZheng}. 
\\~\\
The so-called Stein-Malliavin method was initially introduced in \cite{nourdinpeccati} to establish Berry-Ess\'een bounds for functionals of Gaussian fields: it combines the Malliavin calculus of variations with Stein's method for probabilistic approximations. Since then, these techniques have become increasingly popular and have been extended to the framework of Poisson random measures (see, e.g., \cite{BouPec,PSTU,PecZheng}), as well as to the framework of spectral theory of general Markov diffusion generators (see, e.g., \cite{ledoux}). Quantitative central limit theorems for $U$-statistics based on Poisson point processes were extensively studied in \cite{lesmathias}, where the authors derived a Wiener-It\^o chaos expansion for Poisson based $U$-statistics. More recent results concerning the application of the Stein-Malliavin method in order to prove central limit theorems for needlet-based linear and nonlinear statistics were recently used by \cite{durastanti4,dmp} and \cite{bdmp,bd1}, respectively.

\subsection{Framework and main results}
\noindent $U$-statistics can be introduced as follows. Given a measurable space $\left\{X, \mathfrak{X}\right\}$, let $\left\lbrace X_1,\ldots,X_{\nu} \right\rbrace $ be a collection of independent and identically distributed real-valued random variables over $\left\{X, \mathfrak{X}\right\}$, distributed according to some law $\Prob$: a parameter $\theta$, characterizing the law $\Prob$, is called estimable (or regular, following Hoeffding) of degree $m$, if there exists an real-valued measurable function $h=h\l(x_1,\ldots,x_m\r)$ such that $h\bra{X_1,\ldots,X_m}$ is an unbiased estimator of $\theta$, i.e.,
\begin{equation*}
	\Ex\bra{h\bra{X_1,\ldots,X_m}}=\theta.
\end{equation*}
Without loss of generality, $h$ can be assumed to be symmetric. Indeed, if $g$ is an unbiased estimator for $\theta$ of degree $m$ that is not symmetric, its symmetric version can be easily constructed as follows
\begin{equation*}
h\bra{x_1,\ldots,x_m}=\frac{1}{m!} \sum_{\left(\pi_1\ldots,\pi_m\right) \in \frak{S}_m} g\bra{x_{\pi_1},\ldots,x_{\pi_m}},
\end{equation*} 
where $\frak{S}_m$ is the symmetric group over $m$ elements. Given a sample of observations $\bbra{X_1,\ldots, X_{\nu}}$ of size $\nu>m$, a $U$-statistic with kernel $h$ is given by
\begin{equation*}
	U_m = \frac{1}{\binom{\nu}{m}} \sum_{\bra{i_1,\ldots,i_m}\in\Sigma_{m,\nu}} h\bra{X_{i_1},\ldots,X_{i_m}},
\end{equation*}
where $\Sigma_{m,\nu}$ is the set of all the $\binom{\nu}{m}$ combinations of $m$ integers $i_1< \ldots <i_m$ chosen from $\bbra{1,\ldots,\nu}$. 
\\~\\
\noindent In our framework, we consider a generic compact Riemannian manifold $\bra{\M,g}$ and we denote with $dx$ the Lebesgue measure over $\M$. Take a Poisson process over $\M$, with intensity measure $\mu_t$. In what follows, we assume $\mu_t=R_t f\bra{x} dx$, where $R_t>0$ and $f$ is a density function over $\M$: a complete description of this setting is given in Subsection \ref{sub:steinmalliavin}, cf. also, for instance, \cite{PSTU,PecZheng}. Let us moreover denote by $\bbra{\psi_{j,k}\colon j \geq 0,\ k=1,\ldots, K_j}$ the set of needlets built over $\M$, associated to a scale parameter $B$. $j>0$ is the so-called resolution level, while the index $k=1,\ldots,K_j$ identifies the location of the pixel related to $\psi_{j,k}$ along $\M$. As explained in Subsection \ref{sub:needlet}, for any $j>0$, the area of each pixel, labeled by $\cubew$, and $K_j$ are chosen to be proportional to $B^{-dj}$ and $B^dj$, respectively. Moreover, let $\betac=\betac\bra{f}$ be the corresponding needlet coefficient, containing information on $f$, determined by the inner product 
\begin{equation*}
\betac = \int_{\M}\needlet{x}f(x)dx.
\end{equation*}
Moreover, assume that $f$ belongs to the so-called Besov space $\besov$, which describes the regularity properties of the function $f$ and consequently of the associated wavelet coefficients appearing in the decomposition of $f$.
 Further details concerning the construction and the properties of needlets and needlet coefficients over compact Riemannian manifolds will be given in Subsection \ref{sub:needlet}, see also \cite{gelpes,knp,pesenson}, while Besov spaces will be discussed in Subsection \ref{sub:besov}, cf. also \cite{gm3}. Let us introduce the following notation: for any $\alpha \colon \M \mapsto \reals$, we define the tensor $\alpha^{\otimes n}: \M^n \mapsto \reals$ as  
\begin{equation*}
\alpha^{\otimes n}\bra{\xset{n}}=\prod_{i=1}^{n} \alpha\bra{x_i},
\end{equation*}
where $x_i \in \M$ for $i=1,\ldots n$. Using this notation, let us introduce the kernels
\begin{equation}
\label{eqn:kernel}
h_j\bra{\xset{n}} = \kerdef,
\end{equation}
and their reduced versions, for $p\in \mathbb{N}$, $1\leq p<n$,
\begin{eqnarray}
\label{eqn:kernelred}
h_{j,p}\bra{\xset{p}} &=& \binom{n}{p}\int_{\M^{n-p}}h_j\bra{\xset{n}}  \mu_t^{\otimes (n-p)}\bra{dx_{p+1},\ldots,dx_{n}} \nonumber \\
&=& \sumk \binom{n}{p}R_t^{n-p}\bra{\Mint{\needlet{x}f\bra{x}dx}}^{n-p} \needlettensor{p}{\xset{p}} \nonumber \\
&=& \sumk \binom{n}{p}R_t^{n-p}\betac^{n-p} \needlettensor{p}{\xset{p}},
\end{eqnarray}
taking values respectively over $\M^n$ and $\M^p$. Observe that 
\begin{equation}
\label{defvariance}
\sigma_j^2 := \Var{\sum_{p=1}^{n}I_{p}\left(h_{j,p} \right)} = \sum_{p=1}^{n}\E\left(I_{p}\left(h_{j,p} \right)^2 \right) = \sum_{p=1}^{n}p! R_{t}^{p} \norm{h_{j,p}}_{L^{2}\left( \mu^p\right) }^{2}.
\end{equation}
For all $1 \leq p \leq n$, define $\tilde{h}_{j,p}\bra{\xset{n}} = \sigma_j^{-1}h_{j,p}\bra{\xset{n}}$. Observe that, by construction, $$\E\left( \sum_{p=1}^{n}I_{p}\left(\tilde{h}_{j,p} \right)\right) = 0, \quad \Var{\sum_{p=1}^{n}I_{p}\left(\tilde{h}_{j,p} \right)} = 1.$$ We will study the convergence in law of the random variable 
\begin{equation}
\label{randomtostudy}
\tilde{U}_j := \sum_{p=1}^{n}I_{p}\left(\tilde{h}_{j,p} \right).
\end{equation}
Following \cite{lesmathias}, cf. also Definition \ref{def:ustat}, $\tilde{U}_j$ is a $U$-statistic in the framework of Poisson random measures, see Subsection \ref{sub:steinmalliavin} for further details. 
\begin{remark}
A key aspect of the upcoming analysis is to obtain an asymptotic equivalent of the variance of \eqref{randomtostudy}, namely,  $\sigma_j^2$, or at least an asymptotic lower bound. At first glance, it seems like many possible cases could arise, depending on which chaoses dominate. For instance, for a $U$-statistic of order 10, the chaoses of order 3 and 7 could be the dominating ones, or just the first chaos, depending on the structure of the kernel of the $U$-statistic. It turns out, as proven in Proposition \ref{lowerboundvarianceasymp} (see also Remark \ref{remarkafterproofofpropchaoses}), that only three cases can arise: either the first chaos dominates, or the last chaos dominates, or they must all be asymptotically equivalent. This surprising phenomenon was quite unexpected and contributes to the general understanding of $U$-statistics and, although crucial in our proofs, is hence of independent interest of the rest of the paper. An (rather hidden) illustration of this fact already appears in \cite{bdmp} where $U$-statistics of order two are studied: it was determined that two cases could arise, namely, either the first chaos dominates or the second chaos dominates. In this setting, the second chaos being also the last one, this is in line with the above principle. The third case, although not mentioned in \cite{bdmp} (as it had not been identified), where all chaoses are asymptotically equivalent can be included either in the first chaos domination case or the second chaos domination case. 
\end{remark}
\noindent Finally, let us introduce the following notation: let $\left\lbrace x_j \colon j \geq 0\right\rbrace $ and $\left\lbrace y_j \colon j \geq 0\right\rbrace $ be two positive real-valued sequences and let $C$ denote a generic positive constant. 
\begin{itemize}
\item[--] $x_j \sim y_j \quad \Leftrightarrow \quad \frac{x_j}{y_j} = O(1)$.
\item[--] $x_j \lesssim y_j \quad \Leftrightarrow \quad \underset{j \rightarrow \infty}{\lim} x_j \leq C\underset{j \rightarrow \infty}{\lim} y_j$.
\item[--] $x_j \gtrsim y_j \quad \Leftrightarrow \quad y_j \lesssim x_j$.
\end{itemize}
\subsubsection{Main results: Poissonized case}
The main result of this paper, presented below, is a central limit theorem for $\widetilde{U}_{j}(t)$ for which bounds on the Wasserstein distance (see Definition \ref{def:wasserstein}) between $\widetilde{U}_{j}(t)$ and a standard normal distribution are derived in terms of the dimension $d$ of the manifold $\M$, the expected number of observations $R_t$, the scale parameter of the needlet frame $B$, the resolution level $j$, as well as the Besov regularity parameter $s$ of the density function $f$. 
\begin{theorem}
	\label{maintheoremconvergenceinlaw}
Let $\tilde{U}_j$ be the random variable appearing in \eqref{randomtostudy} and let $Z$ denote a random variable with the $\mathcal{N}(0,1)$ distribution. Then, 
	\begin{enumerate}
		\item[(i)] If $R_t B^{-j(2s+d)} \rightarrow \infty$ as $ j\rightarrow \infty$, then
		\begin{eqnarray*}
			d_W\left(\tilde{U}_j,Z \right) & \lesssim &  R_t^{-\frac{1}{2}}B^{js};
		\end{eqnarray*}
Furthermore, if $R_t^{-\frac{1}{2}}B^{js} \rightarrow 0$ as $ j\rightarrow \infty$, then $\tilde{U}_j$ converges in distribution to $Z$ with a rate given by the above bound.
		\item[(ii)] If $R_t B^{-j(2s+d)} \rightarrow 0$ or $R_t B^{-j(2s+d)} = O(1)$ and $R_tB^{-jd} \rightarrow \infty$ as $ j\rightarrow \infty$, then
		\begin{eqnarray*}
			d_W\left(\tilde{U}_j,Z \right) & \lesssim &   B^{-j\frac{d}{2}}.
		\end{eqnarray*}
Furthermore, $\tilde{U}_j$ converges in distribution to $Z$ with a rate given by the above bound.
	\end{enumerate}
\end{theorem}

\subsubsection{Main results: de-Poissonized case}

\noindent In what follows, we show how Theorem \ref{maintheoremconvergenceinlaw} can be extended to include the case of classical (de-Poissonized) $U$-statistics. This result will be based on a de-Poissonization lemma proved in \cite[Lemma 1.1]{bdmp}, that we will restate here for the sake of convenience. The framework we place ourselves in is the following: let $X = \{X_i \colon i\geq 1\}$ be a sequence of independent and identically distributed random variables with values in $\M$. Furthermore, let $\{N_m \colon m\geq 1\}$ be a sequence of Poisson random variables independent of $X$, such that $N_m$ has a Poisson distribution with mean $m$ for every $m$. For every $m\geq n$ (recall that $n$ is the order of the $U$-statistics under consideration), the symbol $U_m$ denotes the Poissonized $U$-statistic with kernel given by \eqref{eqn:kernel}:
\begin{equation}
\label{poiustat}
U_m = \sum_{1\leq i_1,\ldots, i_n\leq N_m} h_j(X_{i_1},\ldots, X_{i_n}),
\end{equation}
where the sum runs over al $n$-tuples $(i_1, \ldots, i_n)$ such that $i_j\neq i_k$ for $j\neq k$. For every $m\geq 1$, the symbol $U'_m$ denotes the classical $U$-statistic
\begin{equation}
\label{depoiustat}
U'_m = \sum_{1\leq i_1, \ldots, i_n\leq m} h_j(X_{i_1},\ldots, X_{i_n}),
\end{equation}
where, as before, the sum runs over al $n$-tuples $(i_1,\ldots, i_n)$ such that $i_j\neq i_k$ for $j\neq k$. Note that $\mathbb{E}\left( U_m\right)  = \mathbb{E}\left( U'_m\right)  = 0$. The following lemma is the de-Poissonization procedure introduced in \cite[Lemma 1.1]{bdmp}.
\begin{lemma}[Quantitative de-Poissonization lemma]
\label{depoisslemma}
Assume that $\mathbb{E}\left( U_m^2\right) \to 1$ as $m \to \infty$. Then, as $m \rightarrow \infty$, $\mathbb{E}\left( U_m'^2\right)  \to 1$ and
$$
\mathbb{E}[(U_m - U'_m)^2] = O(m^{-1/2}).
$$
\end{lemma} 
\noindent The following result extends Theorem \ref{maintheoremconvergenceinlaw}.
\begin{theorem}
	\label{maintheoremconvergenceinlawdepoissonized}
Let $\widetilde{U'}_m$ be the renormalized version of the random variable appearing in \eqref{depoiustat} and let $Z$ denote a random variable with the $\mathcal{N}(0,1)$ distribution. Then,
	\begin{enumerate}
		\item[(i)] If $m B^{-j(2s+d)} \rightarrow \infty$ as $ m\rightarrow \infty$, then
		\begin{eqnarray*}
			d_W\left(\widetilde{U'}_m,Z \right) & \lesssim &  m^{-\frac{1}{2}}B^{js} + m^{-\frac{1}{4}};
		\end{eqnarray*}
Furthermore, if $m^{-\frac{1}{2}}B^{js} \rightarrow 0$ as $ j\rightarrow \infty$, then $\widetilde{U'}_j$ converges in distribution to $Z$ with a rate given by the above bound.
		\item[(ii)] If $m B^{-j(2s+d)} \rightarrow 0$ or $m B^{-j(2s+d)} = O(1)$ and $mB^{-jd} \rightarrow \infty$ as $ j\rightarrow \infty$, then
		\begin{eqnarray*}
			d_W\left(\widetilde{U'}_m,Z \right) & \lesssim &   B^{-j\frac{d}{2}} + m^{-\frac{1}{4}}.
		\end{eqnarray*}
Furthermore, $\widetilde{U}_j$ converges in distribution to $Z$ with a rate given by the above bound.
	\end{enumerate}
\end{theorem}

\subsection{Applications to Cosmology}\label{applications}
\noindent This section presents some statistical applications meant to provide an applicative background for the results stated here. Both the examples focus on astrophysical problems, in particular on the so-called Cosmic Microwave Background (CMB) radiation, a topic of premier importance for cosmologists in order to understand the origin of the Universe. Roughly speaking, the CMB radiation can be thought of as an isotropic and homogenous thermal radiation generated at the beginning of the Universe. From a mathematical point of view, it can described as a realization of a Gaussian random field over the sphere. The CMB radiation was extensively analysed by spatial missions such \textit{WMAP} or \textit{Planck}, which have collected huge datasets on full-sky fluctuations of the CMB radiation. These measurements are often corrupted by the presence of galactic foregrounds and point sources, such as thermal radiation from dust particles generated by our galaxy and extragalactic point-like microwave signals coming from galaxies and cluster of galaxies among others. As a consequence, in the recent years, a lot of statistical methods were developed in order to recover the original CMB signal. Note that measurements of the CMB radiation are strongly affected by a strong noise due to the presence of the Milky Way, which totally masks the actual CMB signal along the celestial equator. For this reason, this subregieon is called ``masked region''. Using needlets allows us to discard the pixels corresponding to the masked region and only focus on the pixels where the CMB signal is relevant. There exists a constantly growing literature about this topic: we suggest \cite{cama,ghosh,pietrobon2} and the textbooks \cite{dode2004, MaPeCUP} for further details and discussions.  
\subsubsection{Global thresholding} One of the main issues related to the CMB radiation is the estimation of its power spectrum density function. Undergoing investigations on this topic use the so-called global thresholding techniques, which can be summarized as follows. Start by considering the standard regression problem: let $\l\{X_i,Y_i\r\}$, $i=1,\ldots,n$, be independent observational pairs so that $\l\{X_i\r\}$ are uniformly distributed locations over $\M$ (in this case, the 2-dimensional sphere), each of them related to the corresponding $Y_i$ by the following regression formula:
\begin{equation*}
	Y_i=f\bra{X_i}+\varepsilon_i.
\end{equation*}
The function $f \colon \M\mapsto \reals$ is the so-called regression function, while $\l\{\varepsilon_i \colon i=1,\ldots,n \r\}$ is a set of independent zero-mean random variables which can be viewed as observational noise. The aim of this problem is to estimate $f$ from the observational pairs $\l\{X_i,Y_i\r\}$. Among several techniques developed for this purpose, we focus on the so-called global thresholding needlet method, cf. \cite{durastanti6}. In this case, an estimator for $f$ is given by
\begin{equation*}
	\widehat{f}\bra{x} = \sum_{j}\tau_j\sum_{k} \widehat{\beta}_{j,k} \needlet{x}, \quad x\in \M,
\end{equation*}   
where $\widehat{\beta}_{j,k}$ are the empirical needlet coefficients, which act as unbiased estimators of the theoretical ones.
The object $\tau_j$ is called the thresholding function: its purpose is to establish if the set of empirical coefficients corresponding to a given resolution level $j$ reproduces $f$ properly or if it is too affected by the presence of noise. In the first case, the resolution level is kept by the selection procedure, otherwise it is discarded. From a practical point of view, the function $\tau_j$ compares a given $U$-statistic of order $p$ to a threshold proportional to the size of the pixel surrounding the needlet, i.e., $B^{dj}$, and the value $n^{-p/2}$. More specifically, we define 
\begin{equation}\label{eqn:theta}
	\widehat{\Theta}_j\bra{p} = \binom{n}{p}^{-1} \sum_{k=1}^{K_j}\sum_{i_1,\ldots,i_p \in \Sigma_{p,n}} \prod_{m=1}^{p}Y_{i_m} \needlet{X_{i_m} }, 
\end{equation}
whose expectation is given by 
\begin{equation*}
	\Ex \bra{\widehat{\Theta}_j\bra{p}}  = \sum_{k=1}^{K_j}\betac^p
\end{equation*}
while the threshold function is given by
\begin{equation*}
 \tau_j = \ind{\abs{\widehat{\Theta}_j\bra{p}}\geq \frac{B^{dj}}{n^{\frac{p}{2}}}},
\end{equation*}
see, e.g., \cite{durastanti6}.
In the nonparametric estimation framework, the performance of an estimator is typically evaluated by means of its $L^p$-risk, i.e., $\Ex\norm{\widehat{f}-f}_{\Lp}^p$, which measures the discrepancy between the regression function and its estimator. Note that the optimal convergence rate for the $L^p$-risk in the global thresholding framework by needlet methods is given by $n^{\frac{-sp}{2s+d}}$. Applying the results developed here leads to a more precise understanding of the asymptotic behaviour of the estimator. Indeed, they allow to estimate how and how fast each $U$-statistic $\widehat{\Theta}_j\bra{p}$ attains its expectation, providing therefore information about the $p$th moment of the wavelet decomposition, for any resolution level $j$. Note that, even if the structure of \eqref{eqn:theta} seems to be different from \eqref{eqn:kernel}, this problem can be reformulated in terms of density estimation, where we use a sample of size $n$ of observations $\bbra{X_1,\ldots,X_n}$, independent and identically distributed with density $f$. In this case, the thresholding function is given by 
\begin{equation*}
\widehat{\Theta}^\prime_j\bra{p} = \binom{n}{p}^{-1} \sum_{k=1}^{K_j}\sum_{i_1,\ldots,i_p \in \Sigma_{p,n}} \prod_{m=1}^{p} \needlet{X_{i_m} }, 
\end{equation*}
\subsubsection{Point sources detection} Another relevant topic related to the CMB radiation refers to the so-called sparse component separation. As already mentioned, several sources contribute to corrupt the signal measured by CMB experiments, such as, among others, thermal dust, spinning dust, galaxies and clusters. Various techniques were developed in order to recover the original CMB signal, giving rise to the research area of component separation, cf., for instance, \cite{axel,bobsta,scodeller}. In particular, we are interested on the component due to the so-called point sources, the contribution of which becomes of growing importance at higher frequencies in the infrared. Furthermore, the locations of point sources are random within the full-sky and, therefore, they can be described by means of a Poisson process. More specifically, let $A \subset \sphere$ and let $\xi_1,\ldots,\xi_P$ be locations on the sphere: their contribution to the measured power spectrum function is given in terms of 
\begin{equation*}
	N^{(p.s.)}_t\bra{A}=\sum_{p=1}^P N_{t}^{\bra{p}}\delta_{\xi_p}\bra{A},
\end{equation*}
where the label p.s. stands for point sources, $\delta_{\xi_p}$ is the Dirac mass at $\xi_p$ and the mapping $t \mapsto N_t^{\bra{p}}$ is an independent Poisson process with control measure $\lambda_p$ over $\left.\left[0,\infty\right.\right)$, see also \cite{dmp}. In practice, point sources add nonlinearity terms to the original CMB signal, supposed to be Gaussian: as usual, in the presence of nonlinearity, first and second order statistical methods can prove inadequate to perform any analysis and, for this reason, higher order statistical methods have been developed, see again \cite{axel,bobsta}. On the other hand, $U$-statistics can be defined to construct estimators of the polyspectra of order $k$ (namely, the $\l(k+1\r)$-point correlation function in harmonic space) associated to point sources. The techniques  presented here can be usefully applied, for instance, to establish the rate of convergence of these estimators to the bispectrum related to the infrared components.   
More specifically the normalized needlet bispectrum is given by
\begin{equation*}
I_{j_{1},j_{3},j_{3}}= \sum_{k_3}\frac{\widehat{\beta}_{j_1,k_1}}{\operatorname{Var}\bra{\widehat{\beta}_{j_1,k_1}}}\frac{\widehat{\beta}_{j_2,k_2}}{\operatorname{Var}\bra{\widehat{\beta}_{j_2,k_2}}}\frac{\widehat{\beta}_{j_3,k_3}}{\operatorname{Var}\bra{\widehat{\beta}_{j_3,k_3}}} \delta_{j_1,j_2,j_3}\bra{k_1,k_2,k_3}h_{j_1,j_2,j_3}\bra{k_1,k_2,k_3},
\end{equation*}
where the functions $\delta_{j_1,j_2,j_3}\bra{k_1,k_2,k_3}$ and $h_{j_1,j_2,j_3}\bra{k_1,k_2,k_3}$ fix an explicit relationship among the set of indexes $j_1,j_2,j_3,k_1,k_2,k_3$. For any $j \geq 0$, the needlet bispectrum corresponding to the level $j$ can be explicitly symmetrized as follows
\begin{equation*}
	I_{j,j,j}=\binom{n}{3}^{-1}\sum_{\bbra{i_1,i_2,i_3} \in \Sigma_{3,n}} \sum_{k_3}\frac{\psi_{j,k_1}\bra{X_{i_1}}}{\operatorname{Var}\bra{\widehat{\beta}_{j,k_{1}}}}\frac{\psi_{j,k_{2}}\bra{X_{i_2}}}{\operatorname{Var}\bra{\widehat{\beta}_{j,k_2}}}\frac{\psi_{j,k_{3}}\bra{X_{i_3}}}{\operatorname{Var}\bra{\widehat{\beta}_{j,k_3}}} \delta_{j,j,j}\bra{k_1,k_2,k_3}h_{j,j,j}\bra{k_1,k_2,k_3},
\end{equation*}
whose asymptotic behaviour can be studied by means of Theorem \ref{maintheoremconvergenceinlawdepoissonized}.
\subsection{Plan of the paper}\label{sub:plan}
The paper is organized as follows: Section \ref{sec:prel} provides some preliminary background on Poisson random measures, the Stein-Malliavin method, $U$-statistics, as well as on needlet frames and Besov spaces. Section \ref{sec:proof} contains a more detailed statement of the main results as well as their proofs. In Section \ref{interpretationandcomparison}, we offer some interpretations of our main results as well as a parallel with other results on related topics. Finally, Section \ref{sec:auxiliary} presents some auxiliary results needed in Section \ref{sec:proof}, some of which also have their own independent interest.

\section{Preliminaries}\label{sec:prel}
\subsection{Poisson random measures and $U$-statistics}
Let us now introduce the concept of Poisson random measure over a general measure space.
\begin{definition}[\textbf{Poisson random measure}]\label{def:poisson}
	Let $\left(A, \mathcal{A}, \zeta \right)$ be a measure space and $\zeta$  be a $\sigma$-finite, non-atomic measure. A Poisson random measure on $A$ with control measure (or intensity measure) $\zeta$ is a collection of random variables taking values in $\mathbb{Z}^+ \cup \left\{+\infty\right\}$, labeled by $\left\{N\bra{C} \colon C \in \mathcal{A}\right\}$, such that the following properties hold:
	\begin{enumerate}
		\item $N\bra{C}$ has a Poisson distribution of parameter $\zeta \bra{C}$ for every $C\in\mathcal{A}$;
		\item if $C_1,\ldots,C_n \in \mathcal{A}$ are pairwise disjoint, then $N\bra{C_1},\ldots,N\bra{C_n}$ are independent.
	\end{enumerate}
\end{definition}
We recall here the general definition of a Poisson based $U$-statistic given in \cite{lesmathias}, stressing that from now on, for $p \geq2$,  $L^p_s\bra{\zeta^k}\subset L^p\bra{\zeta^k}$ will denote the subspace of symmetric functions.
\begin{definition}[\textbf{Poisson $U$-statistics}]\label{def:ustat}
	Consider a Poisson random measure $N$, with $\sigma$-finite non-atomic intensity measure $\zeta$ on $\bra{A, \mathcal{A}}$ and fix $k \geq 1$. The random variable $F$ is a $U$-statistic of order $k$ based on $N$ if there exists a kernel $h \in L^1_s\bra{\zeta^k}$ such that
	\begin{equation}
	\label{e:ustat}
		F=\sum_{\bra{x_1,\ldots,x_k}\in N^k_{\neq}}h\bra{x_1,\ldots,x_k}.
	\end{equation}	
\end{definition} 
\noindent The following crucial fact is proved by Reitzner and Schulte in \cite[Lemma
3.5 and Theorem 3.6]{lesmathias}.
\begin{proposition}
	\label{p:zap} Consider a kernel $h\in L_{s}^{1}(\zeta^{k})$ such that the
	corresponding $U$-statistic $F$ in \eqref{e:ustat} is square-integrable.
	Then, $h$ is necessarily square-integrable, and $F$ admits a chaotic
	decomposition of the form 
	\begin{equation*}
	F=\int_{A^{k}}h\left( x_{1},\ldots ,x_{k}\right) d\zeta^{k}+\sum_{i=1}^{\infty
	}I_{i}\left( h_{i}\right) ,
	\end{equation*}%
	with 
	\begin{equation*}
	h_{i}(x_{1},\ldots ,x_{i})=\binom{k}{i}\int_{A^{k-i}}h(x_{1},\ldots
	,x_{i},x_{i+1},\ldots ,x_{k})d\zeta^{k-i},\quad (x_{1},\ldots ,x_{i})\in
	A^{i},
	\end{equation*}
	for $1\leq i\leq k$, and $h_{i}=0$ for $i>k$. In particular, $h=h_{k}$ and
	the projection $h_{i}$ is in $L_{s}^{2}(\zeta ^{i})$ for each $1\leq i\leq k$.
\end{proposition}

\subsection{Stein-Malliavin bounds}\label{sub:steinmalliavin}
In this section, we will provide a quick overview on Stein-Malliavin bounds for Poisson random measure: the reader is referred to \cite{PSTU, PecZheng} for further details and discussions. 
\begin{definition}[\textbf{Wasserstein distance}]\label{def:wasserstein}
Let $X$ and $Y$ be two real-valued random variables. The Wasserstein distance between the laws of $X$ and $Y$ is defined to be
\begin{equation*}
d_{W}(X,Y) = \underset{h \in \operatorname{Lip}(1)}{\sup}\abs{\Ex\left[h(X) \right] - \Ex\left[h(Y) \right] },
\end{equation*}
where $\operatorname{Lip}(1)$ denotes the class of real-valued Lipschitz functions, from $\reals$ to $\reals$, with Lipschitz constant less or equal to one. 
\end{definition}
\noindent Recall that the topology induced by the Wasserstein metric on the class of probability measures on $\reals$ is strictly stronger than the topology of weak convergence.
~\\~\\
We now introduce the so-called star-contraction between $g \in L^2\bra{\zeta^p}$ and $h \in L^2\bra{\zeta^m}$. The so-called star-contraction operator $\star_r^{\ell}$ reduces the number of variables in the tensor product between $g$ and $h$ from $p+m$ variables to $p+m-r-\ell$ variables  by identifying $r$ variables in $g$ and $h$ and integrating with respect to $\ell$ among them. More precisely,
\begin{definition}\label{def:starcontraction}
	For $p,m \geq 1$, $r=1,\ldots,p \wedge m$ and $\ell=1,\ldots,r$, given the symmetric functions $g \in L^2\bra{\zeta^p}$ and $h \in L^2\bra{\zeta^m}$, the function $g \star_r^{\ell} h\in L^2\bra{\zeta^{p+m-r-\ell}}$ is the called star-contraction of index $\bra{r,\ell}$ between $g$ and $h$ and is defined as follows
	\begin{eqnarray*}
		&& g \star_r^{\ell} h\bra{t_1,\ldots,t_{p-r},\gamma_1,\ldots,\gamma_{r-\ell},s_1,\ldots,s_{m-\ell}} := \int_{A^{\ell}} g\bra{t_1,\ldots,t_{p-r},\gamma_1,\ldots,\gamma_{r-\ell},z_1,\ldots,z_{\ell}} \\
		&&    \qquad\qquad\qquad\qquad\qquad\qquad\qquad\qquad\qquad h\bra{s_1,\ldots,s_{m-r},\gamma_1,\ldots,\gamma_{r-\ell},z_1,\ldots,z_{\ell}}\zeta^{\ell}\left(dz_1 \cdots dz_{\ell} \right) .
	\end{eqnarray*}
\end{definition} 
~\\
\noindent From now on, we will consider $A = \reals^+\times \M$ and $\mathcal{A}=\mathcal{B}\bra{\reals^+\times \M}$, the class of Borel subsets of $\reals^+\times \M$. We will denote by $N$ the Poisson random measure on $\reals^+\times \M$ with intensity given by the product $\zeta=\rho \times \mu$. The measure $\rho$ is proportional to the Lebesgue measure $\ll$ on $\reals$, i.e., $\rho\bra{ds}=R \cdot \ll\bra{ds}$. $R$ is a fixed parameter, such that $\rho\bra{\sbra{0,t}}=R\cdot t =: R_t$. Note that $t$ can be viewed as the time, cf. also \cite{dmp,bdmp}. On the other hand, $\mu$ describes a probability measure over $\M$ absolutely continuous with respect to the Lebesgue measure over $\M$, so that we can rewrite $\mu\bra{dx}=f\bra{x}dx$. Therefore, for any $t>0$, the object $N_t$ will denote a Poisson  measure over $\bra{\M, \mathcal{B}\bra{\M}}$ with control $\mu_t:= R_t \mu$ such that for any $B \in \mathcal{B}\bra{\M}$, it holds
\begin{equation*}
	\label{eqn:poisson}
	N_t\bra{B}:= N\bra{\sbra{0,t}\times B}.
\end{equation*}
~\\
\noindent The following result is taken from \cite{lrp} and provides a bound on the Wasserstein distance between a vector of multiple Poisson integrals and the multidimensional Gaussian distribution with a given covariance matrix.
\begin{theorem}[Lachi\`eze-Rey, Peccati -- 2013] 
	\label{lachiezereypeccati}
	Let $\tilde{U}_j$ be the random variable appearing in \eqref{randomtostudy} and let $Z$ denote a random variable with the $\mathcal{N}(0,1)$ distribution. Then, there exists a universal constant $C_0 = C_0(n)\in (0,\infty)$, depending uniquely on $n$, such that
	\begin{eqnarray*}
		d_W\left(\tilde{U}_j,Z \right) & \leq & C_0(n)\left\{ \max_1 \norm{\tilde{h}_{j,p}\star_r^\ll \tilde{h}_{j,p}}_{L^2\left( \mu_t^{2p -r-\ell}\right) } + \max_2 \norm{\tilde{h}_{j,p}\star_r^\ll \tilde{h}_{j,q}}_{L^2\left( \mu_t^{p+q -r-\ell}\right) } \right.  \\
		&& \left.  \qquad\qquad\qquad\qquad\qquad\qquad\qquad\qquad\qquad\qquad\qquad\qquad + \max_{p=1,\ldots ,n} \norm{\tilde{h}_{j,p}}_{L^4\left( \mu_t^{p}\right) }^2\right\},
	\end{eqnarray*}
	where $\displaystyle{\max_1}$ ranges over all $2\leq p  \leq n$, and all pairs $(r,\ell)$ such that $r\in\{1,\ldots ,p\}$ and $1\leq \ell \leq r\wedge (p-1)$, whereas $\displaystyle{\max_2}$ ranges over all $1\leq p<q \leq n$ and all pairs $(r,\ell)$ such that $r\in\{1,\ldots ,p\}$ and $\ell\in\{1,...,r\}$.
\end{theorem}

\subsection{Needlets on compact manifolds}\label{sub:needlet}
This subsection is concerned with the construction of needlet systems over compact homogeneous manifolds, following \cite{knp,gelpes} (cf. also \cite{pesenson}).  
Let $\left( \M,g \right)$ be a smooth compact homogeneous Riemannian manifold of dimension $d$, with no boundaries. Let $\Lm$ be the Laplace operator over $\M$: in many cases, such as the $d$-dimensional sphere, projective spaces and other examples (cf. \cite{gelpes}), it can coincide with $-\Delta_{\M}$, the Laplace-Beltrami operator on $\M$ with respect to $g$. The set $\left\{\eigen_q \colon q \geq 0\right\}$ is the set of eigenvalues associated to the eigenfunctions $\left\{u_q \colon q\geq 0\right\}$ solving, for $ q\geq0$, the second-order differential equation $\bra{\Lm - \eigen_q}u_q =0$. From now on, the eigenvalues will be ordered so that $0<\eigen_1 \leq \eigen_2 \leq  \cdots$.
Let $dx$ be the uniform Lebesgue measure over $\M$. As is well-known in the literature (cf. \cite{gm2,gelpes}), the set of eigenfunctions $\left\{u_q \colon q \geq 0\right\}$ forms a complete orthonormal basis for the function space $L^2\bra{\M}=L^2\bra{\M,dx}$, such that 
\begin{equation*}
	\langle u_q, u_{q\prime} \rangle_{\Ltwo} = \int_{\M} u_q\bra{x} \cc{u_{q\prime}\bra{x}} dx = \delta_{q}^{q\prime},
\end{equation*}
where $\delta$ is the Kronecker delta function.  Every function $f \in L^2\bra{\M}$ can therefore be decomposed in terms of its harmonic coefficients, given by $a_q = \langle f, u_p\rangle_{\Ltwo}$, so that 
\begin{equation*}
	\sum_{q \geq 0}\abs{a_q}^2=\norm{f}^2_{\Ltwo}. 
\end{equation*}
Therefore, the following harmonic expansion holds in the $L^2$-sense:
\begin{equation*}
	f\bra{x}=\sum_{q\geq 1}a_q u_q \bra{x}, \quad \forall x \in \M.
\end{equation*}
Our aim is to define a wavelet system over $\bra{\M,g}$ describing a tight frame over $\M$. Recall that a frame over $\M$ can be defined as a countable set of functions, i.e., $\bbra{e_i \colon i\geq 0}$, such that, for any $f \in \Ltwo$, 
\begin{equation*}
	c \norm{f}_{\Ltwo} \leq \sum_{i \geq 0} \abs{ \langle f, e_i\rangle_{\Ltwo}}^2 \leq C \norm{f}_{\Ltwo}.
\end{equation*}
A frame is said to be tight if $c=C$, in which case $C$ is referred to as the tightness constant. 
Following \cite{pesenson} and the references therein (see also \cite{gm1,gm2,npw1}), our aim here is to build a frame $\bbra{\psi_{j,k}}$ on $\Ltwo$: fix $Q \in \mathbb{N}$ and consider the set $\mathcal{H}_{Q}$, given by the span of eigenfunctions $u_q$ such that $q\leq Q$, also called the space of band-limited functions on $\M$ with bandwidth $Q$. Define an $\epsilon$-lattice over $\M$, given by a set of points $\bbra{\xi_k}$ on $\M$, which can be viewed as the centers of balls such that: 
\begin{enumerate}
	\item the balls of radius $\epsilon/4$ are disjoint;
	\item the union of the balls of radius $\epsilon/2$ forms a cover of $\M$;
	\item their multiplicity is not greater than a given $N_{\M}<\infty$.
\end{enumerate}
As proved in \cite{gelpes}, see also \cite{knp,pesenson}, given $0<\delta<1$, there exists a constant $c_0$ depending on $\M$ and $\delta$ such that, if $\epsilon=c_0 Q^{-1/2}$, there exists a set of weights $\lambda_{\xi_k}$ associated to the $\epsilon$-lattice $\chi_Q=\bbra{\xi_k}$, $k=1,\ldots, \text{card}\bra{\chi_Q}$,  such that for any $f\in\mathcal{H}_Q$ it holds that
\begin{equation*}
	\Mint{f\bra{x}dx}=\sum_{\xi_k\in\chi_Q}=\lambda_\xi f\bra{\xi_k}.
\end{equation*}
Let us define the scale parameter $B>1$ and the so-called window function $b:\reals \mapsto \reals^+$ enjoying the following crucial properties:
\begin{enumerate}
	\item $b$ has compact support in $\sbra{B^{-1},B}$;
	\item $b \in C^\infty \bra{\reals}$;
	\item the following partition of unity property holds:
	\begin{equation*}
		\sum_{j>1} b^2\bra{cB^{-2j}}, \quad \text{for any } c>1.
	\end{equation*}
\end{enumerate} 
Let $\Lambdaj = \bbra{q: \ll_q\in \sbra{B^{2\bra{j-1}},B^{2\bra{j+1}}}}$, $f\in \Ltwo$ and introduce the notation $P_q\bra{x,y}=u_q\bra{x}\overline{u_q\bra{y}}$. Consider the sequence of projection operators $A_j$ onto $\mathcal{H}_{B^{j+1}}$, given by
\begin{eqnarray*}
	&& A_0\sbra{f} = \Mint{f\bra{x}dx};\\
	&& A_j\sbra{f}\bra{x}  = \Mint{A_j\bra{x,y}f\bra{y}dy}, \quad j\geq 1, 
\end{eqnarray*}
where the kernels $A_j\bra{x,y}$ are defined by
\begin{equation*}
	A_j\bra{x,y}=\sum_{q\in \Lambdaj} b^2\bra{\barg}P_q\bra{x,y}.
\end{equation*}
Consider the so-called needlet operator with kernel $M_j\bra{x,y}$, given by
\begin{equation*}
	M_j\bra{x,y}=\sumq \bfun{\barg}P_q\bra{x,y}.
\end{equation*}
Using the orthonormality property of $\bbra{u_q \colon q \geq 0}$, we get
\begin{equation*}
	A_j\bra{x,y}=\Mint{M_j\bra{x,z}M_j\bra{z,y}dz}.
\end{equation*}
Following \cite{knp}, note that the kernel $M_j$ satisfies $z \mapsto M_j\bra{x,z} \in \mathcal{H}_{B^{2j}}$. Theorem 6.1 in \cite{gelpes} states that if $f_1,f_2 \in \mathcal{H}_n$, then $f_1f_2\in \mathcal{H}_{cn}$, where $c$ is a positive constant. Let $\mathcal{Z}_j = \mathcal{H}_{cB^{2j}}$ and define $K_j=\text{card}\bra{\mathcal{Z}_j}$. The kernel $A_j$ and its action on $\Lm$ can be represented respectively as
\begin{eqnarray*}
	&& A_j\bra{x,y}  = \sumk \lambda_{j,k} M_{j}\bra{x,\xi_{j,k}}M_{j}\bra{\xi_{j,k},y};\\
	&& A_j\sbra{f}\bra{x} = \sumk \sqrt{\lambda_{j,k}} M_j\bra{x,\xi_{j,k}}\Mint{\sqrt{\lambda_{j,k}}M_{j}\bra{\xi_{j,k},y}f\bra{y}dy}.
\end{eqnarray*}
For $x \in \M$, the needlet (scaling) function is given by 
\begin{equation*}
	\needlet{x} = \sqrt{\cubew} M_j\bra{x,\cubep}=\sqrt{\cubew}\sumq \bfun{\barg}P_q\bra{x,\xi_{j,k}}.
\end{equation*}
For $f \in \Ltwo$, $j\geq 0$ and $k=1,\ldots, K_j$, the needlet coefficient corresponding to $\psi_{j,k}$ is given by
\begin{equation*}
	\betac=\langle f, \psi_{j,k} \rangle_{\Ltwo} =\sqrt{\cubew}\sumq \bfun{\barg}a_q u_{q}\bra{\xi_{j,k}},
\end{equation*}
so that the needlet projection of $f$ onto $\mathcal{Z}_j$ can be rewritten as 
\begin{equation*}
	A_j\sbra{f}\bra{x}=\sumk \betac \needlet{x}, \quad \forall x \in \M.
\end{equation*}
Needlets are characterized by relevant concentration properties, both in the frequency and spatial domains. Indeed, each needlet takes values over a compact set of frequencies, namely, $q \in \Lambda_j$, this being a consequence of the compact support property of the weight function $b$. In addition, the upcoming property follows from the differentiability of the weight function $b$ (see, for instance, \cite{gm2,knp}). For any $x \in \M$ and every $\eta \in \mathbb{N}$, there exists $C_{\eta} > 0$ such that 
\begin{equation}
\label{locapropofneedlets}
\abs{\needlet{x}} \leq \frac{C_{\eta}B^{j\frac{d}{2}}}{\left( 1+B^{jd}d_{\M}\left(x,\xi_{j,k} \right) \right)^{\eta} },
\end{equation}
where $d_{\M}$ denotes a geodesic metric on $\M$. This property emphasizes the fact that needlets are not negligible just inside the pixel surrounding the corresponding cubature point of area $\lambda_{j,k}$.
\\~\\
This above inequality allows one to obtain explicit bounds on the $L^p$-norms of needlets (cf. \cite{knp, npw2}), more specifically
\begin{equation}
\label{eqn:Lpnorm}
c_p B^{jd\bra{\frac{1}{2}-\frac{1}{p}}}\leq \norm{\psi_{j,k}}_{\Lp}\leq C_p B^{jd\bra{\frac{1}{2}-\frac{1}{p}}}.
\end{equation}  
~\\
The following lemma claims another result based on the spatial localization property. 

\begin{lemma}\label{lemma:needloc}
	For any $x \in \sphere$, $q \geq 2$,  $k_{i_1} \neq k_{i_2}$, $i_1 \neq i_2=1,\ldots,q$ and  $\eta \geq 2$, there exists $C_\eta>0$ such that
	\begin{equation*}\label{eqn:needloclemma}
		\int_{\M} \prod_{i=1}^q \psi_{j,k_i}\bra{x}dx \leq \frac{C_{\eta}B^{dj\bra{q-1}}}{\bra{1+B^{dj}\Delta}^{\eta\bra{q-1}}},
	\end{equation*}
	where 
	$$
	\Delta=\min _{i_1,i_2 \in \{ 1,\ldots,q\}, i_1\neq i_2} d_\M (\xi_{j,k_{i_1}},\xi_{j,k_{i_2}}).
	$$
\end{lemma}
\noindent This result first appeared in \cite{durastanti6} for needlets over the $d$-dimensional sphere $\sphere$. The original proof can be easily extended to the compact manifold framework and, therefore, is omitted here for the sake of brevity. Note that, as a straightforward consequence, the following result holds.
\begin{corollary}\label{december1963}
		For any $x \in \sphere$, any even natural number $q \geq 2$ and any constants $c_{k_i}>0$, $i=1,\ldots, q$,  it holds that 
\begin{equation*} 
\int_{\M} \sum_{k_1, \ldots, k_q=1}^{K_j} \prod_{i=1}^q c_{k_{i}}\psi_{j,k_i}\bra{x}dx = O\bra{\norm{\psi_{j,k}}_{\Lq}^q}.
\end{equation*}   
\end{corollary}
\begin{proof}
	Using Lemma \ref{lemma:needloc} yields
\begin{equation*} 
\int_{\M} \sum_{k_1, \ldots, k_q=1}^{K_j} \prod_{i=1}^q c_{k_{i}}\psi_{j,k_i}\bra{x}dx =  \sum_{k=1}^{K_j}  c^q_{k} \int_{\M} \psi_{j,k}^q\bra{x}dx+o\bra{\frac{B^{dj}}{\bra{1+B^{dj}\Delta}^{\eta}}}. 
\end{equation*}   
Using \eqref{eqn:Lpnorm} concludes the proof.
\end{proof}

\subsection{Besov spaces on compact manifolds}\label{sub:besov}
In this subsection, we provide an operative definition of Besov spaces in terms of their approximation properties (for further details and discussions, see, for instance, \cite{bkmpAoSb,gm3,WASA}). Let the approximation error obtained by replacing $f \in \Lr$ by $P_k \in \mathcal{H}_k$ be given by
\begin{equation*}
E_k \bra{f,r} = \inf_{P_k \in \mathcal{H}_k}\norm{f-P_k}_{\Lr}.
\end{equation*}
The Besov space $\besov$, of parameters $r\in \sbra{1,\infty}$, $q\in \sbra{1,\infty}$, $s\geq d/r$, is defined as the functional space such that, for any $f \in \besov$, 
\begin{enumerate}
	\item $f\in \Lr$;
	\item $\sum_{k=0}^{\infty}\frac{1}{k}\bra{k^s E_k\bra{f,r}}^q<\infty$.
\end{enumerate}
Following \cite{bkmpAoSb}, by using standard concentration arguments, Condition (2) is equivalent to 
\begin{equation*}
\sum_{j\geq 0} \bra{B^{js}E_{B^j}\bra{f,r}}^q<\infty.
\end{equation*}	
The Besov space norm $\norm{\cdot}_{\besov}$ is defined as follows
\begin{equation*}
\norm{f}_{\besov} = \left\{
\begin{array}{ll}
\norm{f}_{\Lr} + \sbra{\sum_{j\geq 0} B^{qj\bra{s+d\bra{\frac{1}{2}-\frac{1}{r}}}}\bra{\sumk\abs{\betac}^r}^{\frac{q}{r}}}^{\frac{1}{q}} \  \text{if }q<\infty \\
\norm{f}_{\Lr}+\sup_{j}B^{j\bra{s+d\bra{\frac{1}{2}-\frac{1}{r}}}}\bra{\sumk \abs{\betac}^r}^{\frac{1}{r}}\ \ \ \ \ \ \ \ \ \text{if }q=\infty
\end{array},
\right.
\end{equation*}
so that, if $r,q>1$ and $\max\bra{0, \frac{1}{r}-\frac{1}{q}}<\infty$, $f \in \besov$ if and only if $\norm{f}_{\besov}<\infty$. 
\\~\\
The parameters of the the Besov space $\besov$ can be interpreted as follows:
\begin{itemize}
\item Since $f \in \Lr$, for any $j >0$, the set of needlet coefficients $\left\lbrace \betac \colon k=1,\ldots , K_j\right\rbrace $ belongs to the set of $r$-summable sequences $\ell^r\left(\M \right) $;
\item $q$ controls the weighted $q$-norm along the whole scale of coefficients at $j$;
\item $s$ is the smoothness of the decay rate of the $q$-norm of the needlet coefficients across the scale $j$.
\end{itemize}
Given that, straightforward calculations show that $f \in \besov$ if and only if, for every $j\geq1$
\begin{equation*}
\bra{\sumk \bra{\abs{\betac}\norm{\psi_{j,k}}_{\Lr}}^r}^{\frac{1}{r}}=w_j B^{-js},
\end{equation*}
where $w_j \in \ell^q\left(\M \right) $. Using \eqref{eqn:Lpnorm}, we get
\begin{equation}
\label{eqn:besovcoeff}
\bra{\sumk \bra{\abs{\betac}}^r}^{\frac{1}{r}}=O\bra{B^{-j\bra{s+d\bra{\frac{1}{2}-\frac{1}{r}}}}}.
\end{equation}


\section{Main results}\label{sec:proof}

\noindent In the main result of this paper, namely Theorem \ref{maintheoremconvergenceinlaw}, we establish a quantitative central limit theorem with explicit rates of convergence for each possible subcase. The following theorem is an equivalent restatement of Theorem \ref{maintheoremconvergenceinlaw}, allowing one to better understand how the rates of convergence are derived. In this section, we will first prove the forthcoming result and then show how it implies Theorem \ref{maintheoremconvergenceinlaw} and consequently Theorem \ref{maintheoremconvergenceinlawdepoissonized} by using Lemma \ref{depoisslemma}, hence completing the proofs of our main results.
\begin{theorem}
\label{maintheoremconvergenceinlaw2}
Let $\tilde{U}_j$ be the random variable appearing in \eqref{randomtostudy} and let $Z$ denote a random variable with the $\mathcal{N}(0,1)$ distribution. Then,
\begin{enumerate}
\item[(i)] If $R_tB^{-j(2s+d)} \rightarrow \infty$ as $ j\rightarrow \infty$, then
\begin{eqnarray}
\label{firstcaseoftheorem31}
d_W\left(\tilde{U}_j,Z \right) & \lesssim &   \sum_{p=2}^{n}\sum_{r=1}^{p}\sum_{\ell =1}^{r\wedge (p-1)} \left( R_tB^{-j(2s+d)}\right)^{1-p} \left(R_t B^{-jd} \right)^{\frac{\ell -r}{2}} B^{-j\frac{d}{2}} \nonumber \\
&&  + \sum_{q=2}^{n}\sum_{p=1}^{q-1}\sum_{r=1}^{p}\sum_{\ell =1}^{r} \left( R_tB^{-j(2s+d)}\right)^{1-\frac{p+q}{2}} \left(R_t B^{-jd} \right)^{\frac{\ell -r}{2}} B^{-j\frac{d}{2}} \nonumber \\
&& + \sum_{p=1}^{n} \left( R_tB^{-j(2s+d)}\right)^{1-p} \left(R_t B^{-jd} \right)^{-\frac{p}{2}} B^{-j\frac{d}{2}}.
\end{eqnarray}

\item[(ii)] If $R_tB^{-j(2s+d)} \rightarrow 0$ or $R_t^{-1}B^{j(2s+d)} = O(1)$ as $ j\rightarrow \infty$, then
\begin{eqnarray}
\label{secondcaseoftheorem31}
d_W\left(\tilde{U}_j,Z\right) & \lesssim &   \sum_{p=2}^{n}\sum_{r=1}^{p}\sum_{\ell =1}^{r\wedge (p-1)} \left( R_tB^{-j(2s+d)}\right)^{n-p} \left(R_t^{-1}B^{jd} \right)^{\frac{r-\ell}{2}} B^{-j\frac{d}{2}}\nonumber \\
&&   + \sum_{q=2}^{n}\sum_{p=1}^{q-1}\sum_{r=1}^{p}\sum_{\ell =1}^{r}  \left( R_tB^{-j(2s+d)}\right)^{n-\frac{p+q}{2}} \left(R_t^{-1}B^{jd} \right)^{\frac{r-\ell}{2}} B^{-j\frac{d}{2}} \nonumber \\
&& + \sum_{p=1}^{n} \left( R_tB^{-j(2s+d)}\right)^{n-p} \left(R_t^{-1}B^{jd} \right)^{\frac{p}{2}} B^{-j\frac{d}{2}}.
\end{eqnarray}
\end{enumerate}
\end{theorem}
\begin{remark}
Observe that three terms form each summand in the above bounds: $R_tB^{-j(2s+d)}$, $R_t B^{-jd}$ and $B^{-j\frac{d}{2}}$. Note that $R_tB^{-j(2s+d)}$ is proportional to $\frac{R_t}{K_j}\sum_{k=1}^{K_j}\betac^2$. Heuristically, it can be viewed as the sample variance of the $j$-th level of the needlet decomposition of the Poisson random field. The term $B^{-j\frac{d}{2}}$, as in \cite{bdmp}, can be interpreted as the expected number of observations in each pixel. Indeed, 
\begin{equation*}
\Ex\left( \operatorname{Card}\left\lbrace X_i \colon d(X_i, \cubep) \leq B^{-j} \right\rbrace \right) \simeq R_t \int_{d(X_i, \cubep) \leq B^{-j}}f(x)dx \simeq R_t B^{-jd}.
\end{equation*}
The term $B^{-j\frac{d}{2}}$ comes from the use of relation \eqref{eqn:besovcoeff} and can be viewed as a rescaling weight across the frequencies $j$. Following from these arguments, these terms can be called: effective sample variance for $R_tB^{-j(2s+d)}$, effective sample size for $R_t B^{-jd}$ and scaling factor for $B^{-j\frac{d}{2}}$.
\end{remark}
\begin{remark}
	\label{remarkonwhodominates}
As illustrated in the proof of Proposition \ref{lowerboundvarianceasymp}, it is when $R_t B^{-j(2s+d)} \rightarrow \infty$ as $ j\rightarrow \infty$ that the first chaos dominates, and it is when $R_t B^{-j(2s+d)} \rightarrow 0$ as $ j\rightarrow \infty$ that the last chaos dominates. Furthermore, when $R_t B^{-j(2s+d)} = O(1)$ as $ j\rightarrow \infty$, all the chaos contribute equally. Consequently, what differentiates between the two cases is the estimate one has to use for the variance of $\tilde{U}_j$.
\end{remark}

\begin{proof}[Proof of Theorem \ref{maintheoremconvergenceinlaw2}]
Applying Theorem \ref{lachiezereypeccati}, we can write, using Proposition \ref{propocontractions} and Proposition \ref{proponormL4},
\begin{eqnarray*}
d_W\left(\tilde{U}_j,Z \right) & \lesssim &   \sum_{p=2}^{n}\sum_{r=1}^{p}\sum_{\ell =1}^{r\wedge (p-1)} \sigma_{j}^{-2} R_t^{2n-p-\frac{r}{2} +\frac{\ell}{2}}B^{-j\left(s(2n-2p) + d\left( n -p -\frac{r}{2} +\frac{\ell}{2} -\frac{1}{2}\right)  \right) }  \\
&&   + \sum_{q=2}^{n}\sum_{p=1}^{q-1}\sum_{r=1}^{p}\sum_{\ell =1}^{r}  \sigma_{j}^{-2} R_t^{2n-\frac{p}{2}-\frac{q}{2}-\frac{r}{2} +\frac{\ell}{2}}B^{-j\left(s(2n-p-q) + d\left( n -\frac{p}{2}-\frac{q}{2} -\frac{r}{2}+\frac{\ell}{2} -\frac{1}{2}\right)  \right) } \\
&& + \sum_{p=1}^{n} \sigma_{j}^{-2} R_t^{2n-\frac{3p}{2}} B^{-j\left(s(2n-2p) + d\left( n -\frac{3p}{2} -\frac{1}{2}\right)  \right) }.
\end{eqnarray*}
As stated in Remark \ref{remarkonwhodominates}, the asymptotic behavior of $R_t B^{-j(2s+d)}$ is what dictates which is the dominant chaos order in \eqref{randomtostudy}. Based on this fact, in the case (i), where $R_tB^{-j(2s+d)} \rightarrow \infty$ as $ j\rightarrow \infty$, Proposition \ref{lowerboundvarianceasymp} implies that $\sigma_j^{2} \gtrsim R_t^{2n-1} B^{-j\left(s(2n-2) + d(n-2) \right) }$. Careful algebraic computations imply (i). In the case (ii), when $R_t B^{-j(2s+d)} \rightarrow 0$, Proposition \ref{lowerboundvarianceasymp} implies that $\sigma_j^{2} \gtrsim R_t^{n} B^{jd}$ and when $R_t B^{-j(2s+d)} =O(1)$, Proposition \ref{lowerboundvarianceasymp} implies that $\sigma_j^{2} \gtrsim R_t^{n} B^{jd}$ (in fact, we could also take $\sigma_j^{2} \gtrsim R_t^{2n-p} B^{-j\left(s(2n-2p) + d(n-p-1) \right) }$ for any $1 \leq p \leq n$ as in this case, all these quantities are asymptotically equivalent). Careful algebraic computations then imply (ii).
\end{proof}
\noindent We now prove Theorem \ref{maintheoremconvergenceinlaw} by showing how it is implied by Theorem \ref{maintheoremconvergenceinlaw2}.
\begin{proof}[Proof of Theorem \ref{maintheoremconvergenceinlaw}]
Let's start by assuming that $R_tB^{-j(2s+d)} \rightarrow \infty$ as $ j\rightarrow \infty$, which places us in case (i). Then, in the first line of \eqref{firstcaseoftheorem31}, observe that, as $p >1$, the terms $\left( R_tB^{-j(2s+d)}\right)^{1-p}$ converge to zero at a rate given by the slowest term, which occurs for $p=2$ (the smallest $p$). Similarly, as $\ell -r \leq 0$, the terms $\left(R_t B^{-jd} \right)^{\frac{\ell -r}{2}}$ converge to zero except for the cases where $\ell = r$, in which case these terms are constant. The rate for the first line is hence given by $\left( R_tB^{-j(2s+d)}\right)^{-1} \left(R_t B^{-jd} \right)^{0} B^{-j\frac{d}{2}} = R_t^{-1}B^{j(2s+\frac{d}{2})}$. Applying the same logic, the rate for the second line is going to be given by $\left( R_tB^{-j(2s+d)}\right)^{1-\frac{1+2}{2}} \left(R_t B^{-jd} \right)^{\frac{0}{2}} B^{-j\frac{d}{2}} = R_t^{-\frac{1}{2}}B^{js}$. Finally, the rate for the third line is going to be given by $\left( R_tB^{-j(2s+d)}\right)^{1-1} \left(R_t B^{-jd} \right)^{-\frac{1}{2}} B^{-j\frac{d}{2}} = R_t^{-\frac{1}{2}}$. Combining these observations finally yields
\begin{equation*}
d_W\left(\tilde{U}_j,Z \right) \lesssim   R_t^{-1}B^{j(2s+\frac{d}{2})} + R_t^{-\frac{1}{2}}B^{js} + R_t^{-\frac{1}{2}} \lesssim R_t^{-\frac{1}{2}}B^{js}.
\end{equation*}
Finally, assume $R_tB^{-j(2s+d)} \rightarrow 0$ as $ j\rightarrow \infty$. Assume also that $R_tB^{-jd} \rightarrow \infty$ as $ j\rightarrow \infty$, which places us in case (ii). Then, in the first line of \eqref{secondcaseoftheorem31}, observe that, as $p \leq n$, the term $\left( R_tB^{-j(2s+d)}\right)^{n-p}$ is going to converge to zero expect for $p =n$, for which this term is constant, hence giving the rate of convergence of these terms. Similarly, the terms $\left(R_t^{-1}B^{jd} \right)^{\frac{r-\ell}{2}}$ are going to converge to zero expect when $r = \ell$, in which case they are constant and giving the rate of convergence of these terms. So the terms in the first line are going to converge to zero at a rate given by $\left( R_tB^{-j(2s+d)}\right)^{n-n} \left(R_t^{-1}B^{jd} \right)^{\frac{0}{2}} B^{-j\frac{d}{2}} = B^{-j\frac{d}{2}}$. Similarly, the terms in the second line are going to converge to zero at a rate given by $\left( R_tB^{-j(2s+d)}\right)^{n-\frac{n-1+n}{2}} \left(R_t^{-1}B^{jd} \right)^{\frac{0}{2}} B^{-j\frac{d}{2}} = R_t^{\frac{1}{2}}B^{-j(s+d)}$. Finally, in the third line, observe that the terms $\left( R_tB^{-j(2s+d)}\right)^{n-p}$ $\left(R_t^{-1}B^{jd} \right)^{\frac{p}{2}}$ either converge to zero or are constant. So, as $B^{-j\frac{d}{2}}$ converges to zero, an upper bound for the rate of convergence of the terms of the third line is given by $B^{-j\frac{d}{2}}$ (as it is multiplied by constant terms or terms going to zero as well). If we assume that $R_tB^{-j(2s+d)} = O(1)$ as $ j\rightarrow \infty$, the only modification to the previous case is that the rate of convergence to zero of the second line is directly given by $B^{-j\frac{d}{2}}$ as the term $R_tB^{-j(2s+d)}$ converges to a constant. Combining these observations finally yields
\begin{equation*}
d_W\left(\tilde{U}_j,Z \right) \lesssim   B^{-j\frac{d}{2}} +  R_t^{\frac{1}{2}}B^{-j(s+d)} + B^{-j\frac{d}{2}}.
\end{equation*}
Observe that
\begin{equation*}
\frac{R_t^{\frac{1}{2}}B^{-j(s+d)}}{B^{-j\frac{d}{2}}} = \sqrt{R_tB^{-j(2s+d)}} \underset{j \rightarrow \infty}{\rightarrow} 0,
\end{equation*} 
so that the rate $B^{-j\frac{d}{2}}$ converges to zero slower than $R_t^{\frac{1}{2}}B^{-j(s+d)}$, which concludes the proof.
\end{proof}
\noindent We finally prove Theorem \ref{maintheoremconvergenceinlawdepoissonized}.
\begin{proof}[Proof of Theorem \ref{maintheoremconvergenceinlawdepoissonized}]
Observe that, using the triangle inequality, $$d_W\left(\widetilde{U'}_m,Z \right) \leq d_W\left(\widetilde{U'}_m,\widetilde{U}_m \right)  + d_W\left(\widetilde{U}_m,Z \right), $$ where $\widetilde{U}_m$ denotes the normalized version of the Poissonized $U$-statistic \eqref{poiustat}. Applying Lemma \ref{depoisslemma} on the first summand and Theorem \ref{maintheoremconvergenceinlaw} on the second summand concludes the proof.
\end{proof}
\section{Some interpretations and comparisons with other results}
\label{interpretationandcomparison}
\noindent Note that, since $f \in \besov$, the asymptotic behaviour of the $U$-statistic investigated in Theorem \ref{maintheoremconvergenceinlaw} attains Gaussianity in both the cases (i) and (ii). Both the results can be heuristically motivated as follows. On one hand, in the case (i), the number of sampled observations increases faster than the decay of the sample variance over $\M$, enforcing strong correlation. As a consequence, the rate depends explicitly on the effective sample variance weighted by the scaling factor. On the other hand, in the case (ii), the sample variance decreases faster than the growth of $R_t$, so that we need an additional condition to claim Gaussianity, i.e., the effective sample size goes to infinity. It means that, even if $R_t = o\bra{B^{-j\bra{2s+d}}}$, the number of observations sampled inside each pixel has to increase faster than the shrinking of the size of the pixel. In such a case, the rate of convergence to Gaussianity is provided only by the scaling factor.
\\~\\
Let us now compare our results with the ones established in \cite[Theorem I.2]{bdmp}, which prove a quantitative central limit theorem for $U$-statistics of order two built over the needlet frame in the case of a uniform density over $\sphere$. More specifically, the rate of convergence to Gaussianity of these statistics is given by the sum of three terms, namely, $\bra{R_tB^{-jd}}^{-\frac{1}{2}}$, $B^{-j\frac{d}{2}}$ and $R_t^{-\frac{1}{2}}$. As a consequence, asymptotic normality was attained, provided that $\bra{R_tB^{-jd}} \rightarrow \infty$. Note the similarities with the rates of convergence that we obtain in Theorem \ref{maintheoremconvergenceinlaw}. Indeed, the rates of convergence depend on $R_tB^{-j(2s+d)}$, $R_tB^{-jd}$, $R_t^{-\frac{1}{2}}B^{js}$ and $B^{-j\frac{d}{2}}$. 
Observe that this is analogous, in our framework, to the case $r=\infty$ and $s=0$, i.e., the largest sample case. This choice of parameters, on one hand, make the effective sample variance collapse into the effective sample size and, on the other hand, the scaling factor is annihilated by the $L^\infty$-norm of the needlets. Hence, the case (i), where it is the first chaos of the $U$-statistic  decomposition \eqref{randomtostudy} that dominates, corresponds to \cite[Theorem I.2]{bdmp}, while the case (ii) doesn't exist (the effective sample size cannot converge at the same time to infinity and zero).
In case (ii) of Theorem \ref{maintheoremconvergenceinlaw}, it is the last chaos of the $U$-statistic decomposition \eqref{randomtostudy} that dominates, which is this time analogous to the situation presented in \cite[Theorem I.2]{bdmp} where the second chaos dominates (as the order of the $U$-statistic in \cite{bdmp} is two, the last and second chaos are actually the same). The fact that the first chaos sometimes dominates is due to the presence of strong correlation between the components of the $U$-statistic.

\section{Auxiliary results}\label{sec:auxiliary}
\noindent In order to apply Theorem \ref{lachiezereypeccati}, we need the following results. 

\begin{lemma}
\label{propositionvarianceasymp}
Let $\sigma_j^2$ be the quantity defined in \eqref{defvariance}. Then it holds that 
\begin{equation}
\label{generallowerboundforsigma}
\sigma_j^2 \gtrsim \sum_{p=1}^{n}p!  R_t^{2n-p} B^{-j\left(s(2n-2p) + d(n-p-1) \right) }.
\end{equation}
\end{lemma}
\begin{proof}
Recall that $\sigma_j^2$ is given by \eqref{defvariance}. For any $1 \leq p \leq n-1$, it holds that
\begin{eqnarray*}
\norm{h_{j,p}}_{L^{2}\left( \mu^p\right) }^{2} &\sim&  R_t^{2n-2p} \sum_{k_1,k_2 =1}^{K_j} \beta_{j,k_1}^{n-p}\beta_{j,k_2}^{n-p} \left\langle \psi_{j,k_1}^{\otimes p},\psi_{j,k_2}^{\otimes p} \right\rangle_{L^2\left(\mu^p \right) } \\
&\sim&  R_t^{2n-2p} \sum_{k_1,k_2 =1}^{K_j} \beta_{j,k_1}^{n-p}\beta_{j,k_2}^{n-p} \left\langle \psi_{j,k_1},\psi_{j,k_2} \right\rangle_{L^2\left(\mu \right) }^{p}.
\end{eqnarray*}
Using Corollary \ref{december1963}, one obtains
\begin{equation*}
\norm{h_{j,p}}_{L^{2}\left( \mu^p\right) }^{2}  \gtrsim   R_t^{2n-2p} \sum_{k =1}^{K_j} \beta_{j,k}^{2n-2p} \norm{\psi_{j,k}}_{L^2\left(\mu \right) }^{2p} \gtrsim   R_t^{2n-2p} \sum_{k =1}^{K_j} \beta_{j,k}^{2n-2p} \left( \int_{\mathcal{U}_j}
\psi_{j,k} (s)^2f(s)ds\right)^{p},
\end{equation*}
where $\mathcal{U}_j \subset \mathcal{M}$ is defined by
\begin{equation*}
\mathcal{U}_j := \bigcup_{k=1}^{K_j}B\left(\xi_{j,k}, B^{jd} \right),
\end{equation*}
where $B\left(\xi_{j,k}, B^{jd} \right)$ denotes the the trace on $\M$ of the closed ball with center $\xi_{j,k}$ and radius $B^{jd}$. The traces $B\left(\xi_{j,k}, B^{jd} \right)$ correspond to the pixels partitioning the manifold $\M$. Following from the localization property \eqref{locapropofneedlets}, recall that the needlet of indexes $j,k $ is not negligible only over the corresponding pixel. 
\\~\\
Observe that, for any fixed $j \geq 0$, $\mathcal{U}_j$ is a compact subset of $\mathcal{M}$ on which we assumed that the density function $f$ can be bounded from below by a positive constant. This yields
\begin{equation*}
\norm{h_{j,p}}_{L^{2}\left( \mu^p\right) }^{2}  \gtrsim   R_t^{2n-2p} \sum_{k =1}^{K_j} \beta_{j,k}^{2n-2p} \left( \int_{\mathcal{U}_j}
\psi_{j,k} (s)^2 ds\right)^{p} \gtrsim R_t^{2n-2p} \sum_{k=1}^{K_j} \beta_{j,k}^{2n-2p},
\end{equation*}
where the last equivalence is obtained by using \eqref{eqn:Lpnorm}. Exploiting the estimate \eqref{eqn:besovcoeff} yields
\begin{equation*}
 \sum_{k=1}^{K_j}
\beta_{j,k}^{2n-2p}     \sim B^{-j(2n-2p)\left( s+d\left( \frac{1}{2} - \frac{1}{2n-2p}\right) \right) },
\end{equation*}
so that
\begin{equation*}
\norm{h_{j,p}}_{L^{2}\left( \mu^p\right) }^{2} \gtrsim   R_t^{2n-2p} B^{-j(2n-2p)\left( s+d\left( \frac{1}{2} - \frac{1}{2n-2p}\right) \right) } \sim  R_t^{2n-2p} B^{-j\left(s(2n-2p) + d(n-p-1) \right) }.
\end{equation*}
Whenever $p=n$, we have $h_{j,n} = h_{j}$ and it holds that
\begin{equation*}
\norm{h_{j}}_{L^{2}\left( \mu^n\right) }^{2} \sim R_t^{n} \sum_{k_1,k_2 =1}^{K_j} \left\langle \psi_{j,k_1}^{\otimes n},\psi_{j,k_2}^{\otimes n} \right\rangle_{L^2\left(\mu^n \right) } \sim  R_t^{n} \sum_{k_1,k_2 =1}^{K_j}  \left\langle \psi_{j,k_1},\psi_{j,k_2} \right\rangle_{L^2\left(\mu \right) }^{n} \gtrsim   R_t^{n} K_j \sim R_t^{n} B^{jd},
\end{equation*}
where the last two equivalences were obtained using the same arguments as above. Gathering these estimates together yields
\begin{equation*}
\sigma_j^2 \gtrsim  \sum_{p=1}^{n}p!  R_t^{2n-p} B^{-j\left(s(2n-2p) + d(n-p-1) \right) },
\end{equation*}
as claimed.
\end{proof}

\begin{proposition}
\label{lowerboundvarianceasymp}
Let $\sigma_j^2$ be the quantity defined in \eqref{defvariance}. For $1 \leq p \leq n$, define the quantity $\Lambda_{j,p}$ by
\begin{equation*}
\Lambda_{j,p} := R_t^{2n-p} B^{-j\left(s(2n-2p) + d(n-p-1) \right) }.
\end{equation*}  
Then, exactly one of the following assertions holds:
\begin{enumerate}
\item $\Lambda_{j,p}$ is dominated by $\Lambda_{j,1}$ asymptotically for all $2 \leq p \leq n$, in which case $\sigma_j^2 \gtrsim \Lambda_{j,1}$;
\item $\Lambda_{j,p}$ is dominated by $\Lambda_{j,n}$ asymptotically for all $1 \leq p \leq n-1$, in which case $\sigma_j^2 \gtrsim \Lambda_{j,n}$;
\item $\Lambda_{j,p_1}$ is equivalent to $\Lambda_{j,p_2}$ asymptotically for all $1 \leq p_1, p_2 \leq n$, in which case $\sigma_j^2 \gtrsim \Lambda_{j,n}$;
\end{enumerate}
\end{proposition}
\begin{proof}
Assume that all the $\Lambda_{j,p}$, $1 \leq p \leq n$ are asymptotically equivalent. Then assertions (1) and (2) must be false and using \eqref{generallowerboundforsigma} yields the lower bound $\sigma_j^2 \geq \Lambda_{j,n}$. Assume now that not all the $\Lambda_{j,p}$, $1 \leq p \leq n$ are asymptotically equivalent. Then, in the case where $n = 2$, it is clear that either assertion (1) or (2) holds but not both, and the corresponding lower bounds on $\sigma_j^{2}$ hold by \eqref{generallowerboundforsigma}. In the case where $n > 2$, observe that for any  $1 \leq p_1, p_2 \leq n$, it holds that $\Lambda_{j,p_1}\Lambda_{j,p_2}^{-1} = \left( R_t B^{-j(2s+d)}\right)^{p_2 - p_1}$. So if there exist $1 \leq p_1 \neq p_2 \leq n$ such that $\Lambda_{j,p_1} = O\left( \Lambda_{j,p_2}\right) $, then it follows that $R_t B^{-j(2s+d)} = O(1)$, which in turn implies assertion (3), which is impossible if not all the $\Lambda_{j,p}$, $1 \leq p \leq n$ are asymptotically equivalent. So there exists exactly one $1 \leq q \leq n$ such that $\Lambda_{j,p} = o\left( \Lambda_{j,q}\right) $ for all $1 \leq p \leq n$ such that $p \neq q$. Assume that $1 < q < n$. Then for all $1 \leq p \leq n$ such that $p \neq q$, it holds that $\Lambda_{j,p}\Lambda_{j,q}^{-1} = \left( R_t B^{-j(2s+d)}\right)^{q-p} \rightarrow 0$ as $j \rightarrow \infty$. In particular, $\Lambda_{j,1}\Lambda_{j,q}^{-1} = \left( R_t B^{-j(2s+d)}\right)^{q-1} \rightarrow 0$ and $\Lambda_{j,n}\Lambda_{j,q}^{-1} = \left( R_t B^{-j(2s+d)}\right)^{q-n} \rightarrow 0$. As $q-1 >0$ and $q-n <0$, this implies that $R_t B^{-j(2s+d)} \rightarrow 0$ and $R_t B^{-j(2s+d)} \rightarrow \infty$ simultaneously as $ j\rightarrow \infty$, which is impossible. The only possibilities that remain are either assertion (1) or assertion (2).
\end{proof}
\begin{remark}
\label{remarkafterproofofpropchaoses}
Observe that in Proposition \ref{lowerboundvarianceasymp}, for any $1 \leq p \leq n$, $\Lambda_{j,p}$ corresponds to the normalization (up to constants) of the $p$-th chaos term in \eqref{randomtostudy}. It follows from Proposition \ref{lowerboundvarianceasymp} that either the first chaos term $I_{1}\left(\tilde{h}_{j,1} \right)$, the highest chaos term $I_{n}\left(\tilde{h}_{j,n} \right)$ or all chaos terms contribute to the limit of \eqref{randomtostudy} as $j \rightarrow \infty$.
\end{remark}

\begin{proposition}
\label{propocontractions}
Let the above notation prevail. For all $1 \leq p \leq q \leq n$, $1 \leq r \leq p$ and $1 \leq \ell \leq r \wedge (q-1)$, it holds that 
\begin{equation*}
\norm{\tilde{h}_{j,p}\star_r^\ll \tilde{h}_{j,q}}_{L^2\left( \mu_t^{p+q -r-\ell}\right) }^{2} \lesssim \sigma_{j}^{-4} R_t^{4n-p-q-r +\ll}B^{-j\left(s(4n-2p-2q) + d(2n -p-q -r +\ell -1) \right) }.
\end{equation*}
\end{proposition}
\begin{proof}
\noindent Using \eqref{eqn:kernelred} to compute the contractions $\tilde{h}_{j,p} \star_{r}^{\ell} \tilde{h}_{j,q}$ for all $1 \leq p < q \leq n$, $1 \leq r \leq p$ and $1 \leq \ell \leq r $ yields
\begin{eqnarray*}
&& \tilde{h}_{j,p}\star_r^\ll \tilde{h}_{j,q}\left(x_1,\ldots , x_{p-r},\gamma_1, \ldots , \gamma_{r-\ell},y_1,\ldots, y_{q-r} \right)  = \sigma_{j}^{-2}\sum_{k_1,k_2 =1}^{K_j}
\binom{n}{p}\binom{n}{q}R_t^{2n-p-q+\ll}\beta_{j,k_1}^{n-p}\beta_{j,k_2}^{n-q} \\
&& \qquad\qquad  \psi_{j,k_1}^{\otimes (p-r)}\bra{x_1, \ldots ,x_{p-r}}\psi_{j,k_2}^{\otimes (q-r)}\bra{y_1, \ldots ,y_{q-r}} \left[ \psi_{j,k_1}^{\otimes r- \ell}\psi_{j,k_2}^{\otimes r- \ell}\right] \bra{\gamma_1,\ldots,\gamma_{r-\ll}} \left\langle \psi_{j,k_1},\psi_{j,k_2} \right\rangle_{L^2\left(\mu \right) }^\ll.
\end{eqnarray*}
Based on that expression, one can deduce the following estimate
\begin{eqnarray*}
&& \norm{\tilde{h}_{j,p}\star_r^\ll \tilde{h}_{j,q}}_{L^2\left( \mu_t^{p+q -r-\ell}\right) }^{2} = R_t^{p+q-r-\ell} \int_{\M^{p+q-r-\ell}}\left( \tilde{h}_{j,p}\star_r^\ll \tilde{h}_{j,q}\right)^2 d\mu^{\otimes (p+q-r-\ell)} \\
 && \qquad\qquad\qquad \sim    \sigma_{j}^{-4} R_t^{4n-p-q-r +\ll}\sum_{k_1,k_2,k_3,k_4 =1}^{K_j}
\beta_{j,k_1}^{n-p}\beta_{j,k_2}^{n-q}\beta_{j,k_3}^{n-p}\beta_{j,k_4}^{n-q} \left\langle \psi_{j,k_1},\psi_{j,k_3} \right\rangle_{L^2\left(\mu \right) }^{p-r}   \left\langle \psi_{j,k_2},\psi_{j,k_4} \right\rangle_{L^2\left(\mu \right) }^{q-r} \\
&& \qquad \qquad \qquad\qquad\qquad\qquad\qquad\qquad\quad \left\langle \psi_{j,k_1},\psi_{j,k_2} \right\rangle_{L^2\left(\mu \right) }^{\ell} \left\langle \psi_{j,k_3},\psi_{j,k_4} \right\rangle_{L^2\left(\mu \right) }^{\ell} \left\langle \psi_{j,k_1}\psi_{j,k_2},\psi_{j,k_3}\psi_{j,k_4} \right\rangle_{L^2\left(\mu \right) }^{r-\ell}.
\end{eqnarray*}
Using the asymptotic equivalence relation in Corollary \ref{december1963} repeatedly yields
\begin{equation}
\label{firstestimateofcontnorm}
\norm{\tilde{h}_{j,p}\star_r^\ll \tilde{h}_{j,q}}_{L^2\left( \mu_t^{p+q -r-\ell}\right) }^{2}  \sim  \sigma_{j}^{-4} R_t^{4n-p-q-r +\ll}\sum_{k=1}^{K_j}
\beta_{j,k}^{4n-2p-2q} \norm{\psi_{j,k}}_{L^2\left(\mu \right) }^{2p+2q-4r-4\ell} \norm{\psi_{j,k}}_{L^4\left(\mu \right) }^{4(r-\ell)}.
\end{equation}
Using the fact that the density function $f$ is bounded above along with the norm estimates \eqref{eqn:Lpnorm} allows one to write $\norm{\psi_{j,k}}_{L^2\left(\mu \right) }^{2p+2q-4r-4\ell} \leq \norm{\psi_{j,k}}_{L^2\left(\M \right) }^{2p+2q-4r-4\ell} \lesssim 1$ and $\norm{\psi_{j,k}}_{L^4\left(\mu \right) }^{4(r-\ell)} \leq \norm{\psi_{j,k}}_{L^4\left(\M \right) }^{4(r-\ell)} \lesssim B^{jd(r-\ell)}$. Combining these two estimates together in \eqref{firstestimateofcontnorm} yields
\begin{equation}
\label{estimateusedforthencase}
\norm{h_{j,p}\star_r^\ll h_{j,q}}_{L^2\left( \mu_t^{p+q -r-\ell}\right) }^{2}  \lesssim  \sigma_{j}^{-4} R_t^{4n-p-q-r +\ll}B^{jd(r-\ell)}\sum_{k=1}^{K_j}
\beta_{j,k}^{4n-2p-2q}.
\end{equation}
Exploiting the estimate \eqref{eqn:besovcoeff} yields, for any $1 \leq p \leq q \leq n$ such that $p+q \leq 2n-1$,
\begin{equation*}
\label{estimbeta2nminus2p}
\sum_{k=1}^{K_j}
\beta_{j,k}^{4n -2p -2q} \sim B^{-j(4n -2p -2q)\left( s+d\left( \frac{1}{2} - \frac{1}{4n -2p -2q}\right) \right) },
\end{equation*}
so that, for any $1 \leq p \leq q \leq n$ such that $p+q \leq 2n-1$,
\begin{equation*}
\norm{\tilde{h}_{j,p}\star_r^\ll \tilde{h}_{j,q}}_{L^2\left( \mu_t^{p+q -r-\ell}\right) }^{2}  \lesssim  \sigma_{j}^{-4} R_t^{4n-p-q-r +\ll}B^{-j\left(s(4n-2p-2q) + d(2n -p-q -r +\ell -1) \right) }.
\end{equation*}
In the case where $p=q=n$ (in which case $\tilde{h}_{j,p} = \tilde{h}_{j,q} = \tilde{h}_{j,n}=\tilde{h}_{j}$), recalling that $K_j \sim B^{jd}$ and using  \eqref{estimateusedforthencase} implies that
\begin{equation*}
\norm{\tilde{h}_{j,n}\star_r^\ll \tilde{h}_{j,n}}_{L^2\left( \mu_t^{2n -r-\ell}\right) }^{2}  \lesssim  \sigma_{j}^{-4}K_j R_t^{2n-r +\ll} B^{-jd(-r+\ell)} \sim \sigma_{j}^{-4} R_t^{2n-r +\ll}  B^{-jd(-r+\ell -1)},
\end{equation*}
which concludes the proof.
\end{proof}

\begin{proposition}
\label{proponormL4}
Let the above notation prevail. For all $1 \leq p \leq n$, it holds that 
\begin{equation*}
\norm{\tilde{h}_{j,p}}_{L^{4}\left( \mu_t^p\right) }^{4} \lesssim \sigma_{j}^{-4} R_t^{4n-3p} B^{-j\left(s(4n-4p) + d(2n -3p -1) \right) }.
\end{equation*}
\end{proposition}
\begin{proof}
Using \eqref{eqn:kernelred} and the fact that $f$ is bounded from above, one can write
\begin{eqnarray*}
\norm{\tilde{h}_{j,p}}_{L^{4}\left( \mu_t^p\right) }^{4} &\sim  & \sigma_{j}^{-4}R_{t}^{p}\int_{\M^p}h_{j,p}^4 d\mu^{\otimes p}\\
&\sim & \sigma_{j}^{-4} R_{t}^{p} R_t^{4n-4p} \sum_{k_1,k_2,k_3,k_4 =1}^{K_j}\beta_{j,k_1}^{n-p}\beta_{j,k_2}^{n-p}\beta_{j,k_3}^{n-p}\beta_{j,k_4}^{n-p} \left\langle \psi_{j,k_1}^{\otimes p}\psi_{j,k_2}^{\otimes p}, \psi_{j,k_3}^{\otimes p}\psi_{j,k_4}^{\otimes p} \right\rangle_{L^2\left(\mu^p \right) }\\
& \sim &  \sigma_{j}^{-4} R_t^{4n-3p} \sum_{k_1,k_2,k_3,k_4 =1}^{K_j} \beta_{j,k_1}^{n-p}\beta_{j,k_2}^{n-p}\beta_{j,k_3}^{n-p}\beta_{j,k_4}^{n-p}  \left\langle \psi_{j,k_1}\psi_{j,k_2}, \psi_{j,k_3}\psi_{j,k_4}\right\rangle_{L^2\left(\mu \right) }^p.
\end{eqnarray*}
Using Corollary \ref{december1963} 
and the norm estimates \eqref{eqn:Lpnorm}, we obtain
\begin{equation}
\label{estim2onthel4norm}
\norm{\tilde{h}_{j,p}}_{L^{4}\left( \mu_t^p\right) }^{4} \sim \sigma_{j}^{-4} R_t^{4n-3p} \sum_{k =1}^{K_j} \beta_{j,k}^{4n-4p} \norm{\psi_{j,k}}_{L^4\left(\mu \right) }^{4p} \lesssim \sigma_{j}^{-4} R_t^{4n-3p}B^{jdp} \sum_{k =1}^{K_j} \beta_{j,k}^{4n-4p}.
\end{equation}
Exploiting the estimate \eqref{eqn:besovcoeff} yields, for $1 \leq p \leq n-1$,
\begin{equation*}
\sum_{k=1}^{K_j}
\beta_{j,k}^{n-p} \sim B^{-j(4n-4p)\left( s+d\left( \frac{1}{2} - \frac{1}{4n-4p}\right) \right) },
\end{equation*}
so that, for $1 \leq p \leq n-1$,
\begin{equation*}
\norm{\tilde{h}_{j,p}}_{L^{4}\left( \mu_t^p\right) }^{4} \lesssim   \sigma_{j}^{-4} R_t^{4n-3p}B^{jdp} B^{-j\left(s(4n-4p) + d(2n -2p -1) \right) } \sim \sigma_{j}^{-4} R_t^{4n-3p} B^{-j\left(s(4n-4p) + d(2n -3p -1) \right) }.
\end{equation*}
In the case where $p=n$ (in which case $\tilde{h}_{j,n}=\tilde{h}_{j}$), we have directly from \eqref{estim2onthel4norm} that
\begin{equation*}
\norm{\tilde{h}_{j}}_{L^{4}\left( \mu_t^n\right) }^{4} \lesssim   \sigma_{j}^{-4} R_t^{n}K_j B^{jdn} \sim \sigma_{j}^{-4} R_t^{n} B^{jd(n+1)},
\end{equation*}
which concludes the proof.
\end{proof}
\begin{acknow*}
The authors wish to thank Domenico Marinucci for insightful discussions and remarks, as well as for the partial support provided by the ERC grant Pascal n. 277742.
\end{acknow*}
\bibliography{bibliografia}
\end{document}